\documentclass[10pt,a4paper]{amsart}
\usepackage{amsfonts,amsmath,amssymb}
\usepackage{amssymb,amsthm,amsxtra}
\usepackage[usenames]{color}

\usepackage{amscd}
\usepackage{amsthm}
\usepackage{amsfonts}
\usepackage{amssymb}
\usepackage{mathrsfs}

\usepackage{hyperref}
\usepackage[all]{xy}
\newtheorem*{theorem*}{Theorem}
\newtheorem{lemma}{Lemma}[subsection]

\newtheorem{remark}[lemma]{Remark}

\newtheorem*{conjecture*}{Conjecture}

\newtheorem{thm}[lemma]{Theorem}
\newtheorem{prop}[lemma]{Proposition}
\newtheorem{lem}[lemma]{Lemma}
\newtheorem{defn}[lemma]{Definition}
\newtheorem{notn}[lemma]{Notation}
\newtheorem{cor}[lemma]{Corollary}

\newtheorem{rem}[lemma]{Remark}

\newtheorem{introtheorem}{Theorem}

\newtheorem{introprop}[introtheorem]{Proposition}

\baselineskip 18pt \textwidth 16cm  \theoremstyle{plain}

\newcommand{\de}{\delta}

\newcommand{\Hom}{\operatorname{Hom}}
\newcommand{\Irr}{\operatorname{Irr}}
\newcommand{\Gr}{\operatorname{Gr}}

\newcommand{\rr}{\mathbb{R}}

\renewcommand{\Im}{\operatorname{Im}}
\newcommand{\Ker}{\operatorname{Ker}}
\newcommand{\Coker}{\operatorname{Coker}}

\newcommand{\cP}{{\mathcal P}}
\newcommand{\cA}{{\mathcal A}}
\newcommand{\cB}{{\mathcal B}}
\newcommand{\C}{{\mathbb C}}

\newcommand{\Span}{{\operatorname{Span}}}

\newcommand{\cD}{{\mathcal{D}}}
\newcommand{\g}{{\mathfrak{g}}}

\newcommand{\GL}{\operatorname{GL}}

\newcommand{\gl}{{\mathfrak{gl}}}

\newcommand{\Sc}{{\mathcal S}}

\newcommand{\cM}{\mathcal{M}}

\newcommand{\cF}{\mathcal{F}}
\newcommand{\cG}{\mathcal{G}}
\newcommand{\cH}{\mathcal{H}}

\begin{document}

\author{Avraham Aizenbud}
\address{Avraham Aizenbud, Faculty of Mathematics
and Computer Science, The Weizmann Institute of Science POB 26,
Rehovot 76100, ISRAEL.} \email{aizenr@yahoo.com}
\author{Dmitry Gourevitch}
\address{Dmitry Gourevitch,
School of Mathematics,
Institute for Advanced Study,
Einstein Drive, Princeton, NJ 08540 USA}
\email{dimagur@ias.edu}
\date{\today}

\title[Multiplicity free Jacquet modules]{Multiplicity free Jacquet modules}
%
%
%
%
%
%
%
%
%
%
\begin{abstract}
Let $F$ be a non-Archimedean local field or a finite field.
Let $n$ be a natural number and $k$ be $1$ or $2$.
Consider $G:=\GL_{n+k}(F)$ and let $M:=\GL_n(F) \times GL_k(F)<G$ be a maximal Levi subgroup. Let $U< G$ be the corresponding unipotent subgroup and let $P=MU$ be the corresponding parabolic subgroup.
Let $J:=J_M^G: \cM(G) \to \cM(M)$ be the Jacquet functor (i.e. the functor of coinvariants w.r.t. $U$).
In this paper we prove that $J$ is a multiplicity free functor, i.e.
$$\dim \Hom_M(J(\pi),\rho)\leq 1,$$
for any irreducible representations $\pi$ of $G$ and $\rho$ of $M$.

To do that we adapt the classical method of Gelfand and Kazhdan that  proves "multiplicity free" property of certain representations to prove "multiplicity free" property of certain functors.

At the end we discuss whether other Jacquet functors are multiplicity free.
\end{abstract}

\maketitle

\tableofcontents

\section{Introduction}
Let $F$ be a non-Archimedean local field or a finite field.
Let $n$ be a natural number and $k$ be $1$ or $2$.
Consider $G:=\GL_{n+k}(F)$ and let $M:=\GL_n(F) \times GL_k(F)<G$ be a maximal Levi subgroup. Let $U< G$ be the corresponding unipotent subgroup and let $P=MU$ be the corresponding parabolic subgroup.
Let $J:=J_M^G: \cM(G) \to \cM(M)$ be the Jacquet functor (i.e. the functor of coinvariants w.r.t. $U$).
We will fix the notations $F,n,G,M$ and $U$ throughout the paper.

In this paper we prove the following theorem.
\begin{introtheorem} \label{thm:PiRho}
Let $\pi$ be an irreducible representation of $G$ and $\rho$ be an irreducible representation of $M$. Then $$\dim \Hom_M(J(\pi),\rho)\leq 1.$$
\end{introtheorem}
As we will show in \S \ref{sec:ImpRes}, this theorem is equivalent to the following one.
\begin{introtheorem} \label{thm:Sc}
Let $G\times M$ act on $G/U$ by $(g,m)([g'])=[g g' m^{-1}].$ This action is well defined since $M$ normalizes $U$.
Consider the space of Schwartz  measures $\cH(G/U)$ (i.e. compactly supported measures which are locally constant w.r.t. the action of $G$) as a representation of $G\times M$.
Then this representation is multiplicity free, i.e. for any irreducible representation $\pi$ of $G\times M$ we have
$$\dim \Hom_{G \times M}(\cH(G/U),\pi)\leq 1.$$
\end{introtheorem}
By Frobenius reciprocity, this theorem is in turn equivalent to the following one.
\begin{introtheorem} \label{thm:GP}
Consider $P$ to be diagonally embedded in $G \times M$. Then the pair $(G \times M,P)$ is a Gelfand pair i.e. for any irreducible representation $\pi$ of $G\times M$ we have
$$\dim \Hom_{P}(\pi, \C)\leq 1.$$
\end{introtheorem}

Theorem \ref{thm:PiRho} implies also the following theorem.
\begin{introtheorem} \label{thm:GL}
Suppose $k=1$ and let $H=GL_n(F)$ be standardly embedded inside $G$. Let $\pi$ be an irreducible representation of $G$ and $\rho$ be an irreducible representation of $H$.
Then $$\dim \Hom_H(J(\pi),\rho)\leq 1.$$
\end{introtheorem}

We will prove the implications mentioned above between theorems \ref{thm:PiRho}, \ref{thm:Sc}, \ref{thm:GP} and \ref{thm:GL} in \S \ref{sec:ImpRes}.

\subsection{A sketch of the proof}  $ $

Using a version of the Gelfand-Kazhdan criterion we deduce Theorem \ref{thm:Sc} from the following one
\begin{introtheorem}
Any distribution on $(U^t \setminus G) \times (G/U)$ which is invariant with respect to the action of $G\times M$ given by $(g,m)([x],[y]):=([mxg^{-1}],[gym^{-1}])$ is also invariant with respect to the involution $([x],[y]) \mapsto ([y^t],[x^t])$.
\end{introtheorem}
By the method of Bernstein-Gelfand-Kazhdan-Zelevinski (Theorem \ref{thm:BGKZ}) it is enough to prove that the involution preserves all $G \times M$ orbits. This we deduce from the following geometric statement.
\begin{introprop}
 \label{intthm:Geo}
Let $X:=X_{n,k}:=\{A,B\in Mat_{n+k}| AB=BA=0,rank(A)=n,rank(B)=k\}$.  Let $G$ act on $X_{n,k}$ by conjugations.
Define the transposition map $\theta:=\theta_{n,k}:X_{n,k} \to X_{n,k}$ by $\theta(A,B):=(A^t,B^t)$.

Then any $G$-orbit in  $X_{n,k}$ is $\theta$-invariant.
\end{introprop}

We deduce this geometric statement from the key lemma \ref{lem:Key}, which states that every $M$-orbit in $U^t \setminus \GL_k(F)/U$ is transposition invariant, where $M<GL_k(F)$ is a Levi subgroup and $U$ is the corresponding unipotent subgroup. This lemma is a straightforward computation since $k\leq 2$, but for bigger $k$ it is not true.

\subsection{Related problems}
\subsubsection{Case $k=1$} In case when $k=1$ and $F$ is a local field, a stronger theorem holds. Namely, the functor of restriction from $\GL_{n+1}(F)$ to $\GL_n(F)$ is multiplicity free. This is proven in \cite{AGRS} for $F$ of characteristic 0, in \cite{AAG} for $F$ of positive characteristic. It is also proven for Archimedean $F$ in \cite{AG_AMOT,SZ}.

 This stronger statement does not hold for finite fields already for $n=1$.
Theorem \ref{thm:GL} may be viewed as a weaker form of this statement that works uniformly for local and  finite fields. 

Note that in case when $k=1$ and $F$ is a finite field, there is an alternative proof of Theorem \ref{thm:GL} which is based on the classification of irreducible representations of $\GL_n(F)$, see \cite{Fad,Gre,Zel}.

\subsubsection{The Archimedean case} We believe that the analog of Theorem \ref{thm:PiRho} for Archimedean $F$ holds. For $k=1$ it holds as explained above. For $k=2$  we believe that the proof given in this paper can be adapted to the Archimedean case. However this will require additional analysis.

\subsubsection{Higher rank cases} One can ask whether an analog of Theorem \ref{thm:PiRho} holds when $M$ is an arbitrary Levi subgroup of $G$. If $F$ is a local field, we do not know the answer for this question.
If $F$ is a finite field, such  analog of Theorem \ref{thm:PiRho}  holds only in the cases at hand. This is related to the fact that the restriction of any irreducible representation of
the permutation group $S_{n_1+...+n_l} $ to $S_{n_1} \times ... \times S_{n_l}$ is multiplicity free if and only if $l \leq 2$ and $\min(n_1,n_2) \leq 2$. We discuss those questions in \S \ref{sec:HiRank}.




\subsection{Contents of the paper}$ $

In \S \ref{sec:Prel} we give the necessary preliminaries.
In \S\S \ref{subsec:GenNot} we introduce notation that we will use throughout the paper.
In \S\S \ref{subsec:ell} we give some preliminaries and  notation on $l$-spaces,
l -groups and their representations based on \cite{BZ}. In \S\S
\ref{subsec:MultFreeFun} we define multiplicity free functors and formulate two theorems that enable to reduce "multiplicity free" property of a strongly right exact functor between the categories of smooth representations of two $l$-groups to
"multiplicity free" property of a certain representation of the product of those groups. We prove those theorems in Appendix \ref{app:MultFree}. In \S\S \ref{subsec:GK} we formulate a version of Gelfand-Kazhdan criterion for "multiplicity free" property of representations of the form 
$\Sc(X)$. 
We prove this version in  Appendix \ref{app:GK}.
In  \S\S \ref{subsec:BGKZ} we recall a criterion for vanishing of equivariant distributions in  terms of stabilizers of points.
In \S\S \ref{subsec:DelFilt} we recall the Deligne (weight) filtration attached to a nilpotent operator on a vector space.

In \S \ref{sec:ImpRes} we prove equivalence of Theorems \ref{thm:PiRho}, \ref{thm:Sc} and \ref{thm:GP} and deduce Theorem \ref{thm:GL} from them.

In \S \ref{sec:RedGeo} we reduce Theorem \ref{thm:Sc} to the geometric statement.

In \S \ref{sec:PFGeoStat} we prove the geometric statement.

In \S \ref{sec:HiRank} we discuss whether an analog of Theorem \ref{thm:PiRho} holds when $M$ is an arbitrary Levi subgroup.
In \S\S \ref{subsec:Perm} we answer an analogous question for permutation groups. In \S\S \ref{subsec:Con} we discuss the connection between the questions for permutation groups and general linear groups over finite fields. In \S\S \ref{subsec:LocHighRank} we discuss the local field case.

In Appendix \ref{app:MultFree} we prove theorems on strongly right exact functors between the categories of smooth representations of two reductive groups from \S\S \ref{subsec:MultFreeFun}.

In Appendix \ref{app:GK} we prove a version of Gelfand-Kazhdan criterion for "multiplicity free" property of geometric representations from  \S\S \ref{subsec:GK}.

\subsection{Acknowledgments}$ $

We thank {\bf Joseph Bernstein} for initiating this work by telling us the case $k=1$.
We also thank {\bf Joseph Bernstein}, {\bf Evgeny Goryachko} and {\bf Erez Lapid} for useful discussions.

This work was conceived while the authors were visiting 
the Max Planck Institute
fur Mathematik (MPIM) in Bonn. We wish to thank the MPIM for its hospitality.
D.G. also worked on this paper when he was  a post-doc at  the Weizmann Institute of Science. He wishes to thank the Weizmann Institute for wonderful working conditions during this post-doc and during his graduate  studies. 

Both authors were partially supported by a BSF grant, a GIF grant, and an ISF Center
of excellency grant. A.A was also supported by ISF grant No. 583/09 and 
D.G. by NSF grant DMS-0635607. Any opinions, findings and conclusions or recommendations expressed in this material are those of the authors and do not necessarily reflect the views of the National Science Foundation.

\section{Preliminaries} \label{sec:Prel}

\subsection{General notation} \label{subsec:GenNot}

\begin{itemize}


\item For a group $H$ acting on a set $X$ and a point $x \in X$ we denote by $Hx$ or by $H(x)$ the orbit of $x$ and by $H_x$ the stabilizer of $x$.
We also denote by $X^H$ the set of fixed points in $X$.

\item For a representation $V$ of a group $H$ we denote by $V^H$ the space of invariants and by $V_H$ the space of coinvariants, i.e.
$V_H:=V/(\Span\{v - gv\, | \, g\in  H, \, v\in  V\})$.

\item For a Lie algebra $\g$ acting on a vector space $V$ we denote by $V^{\g}$ the space space of invariants. Similarly, for any element $X \in \g$ we denote by $V^X$ the kernel of the action of $X$.

\item For a linear operator $A:V \to W$ we denote the cokernel of A by $\Coker A:= W/ImA$.

\item For a linear operator $A:V \to V$ and an $A$-invariant subspace $U \subset V$ we denote by $A|_U:U \to U$  and $A|_{V/U}:V/U \to V/U$ the natural induced operators.
%
%
\end{itemize}

\subsection{$l$-spaces and $l$-groups} \label{subsec:ell} $ $

We will use the standard terminology of $l$-spaces introduced in \cite{BZ}.
Let us recall it.

\begin{itemize}
\item An $l$-space is a Hausdorff locally compact totally disconnected topological space.

\item For an $l$-space $X$ we
denote by $\Sc(X)$ the space of Schwartz functions on $X$, i.e. locally constant compactly supported) functions on $X$. We denote by $\Sc^*(X)$ the dual
space and call its elements distributions.
\item In \cite{BZ} there was introduced the notion of "l-sheaf". As it was later realized (see e.g. \cite[\S\S 1.3]{Ber_Lec}) this notion is equivalent to the usual notion of sheaf on an $l$-space, so we will use the results of \cite{BZ} for sheaves. 

\item For a sheaf $\cF$ on an $l$-space $X$ we denote by $\Sc(X,\cF)$ the space of compactly supported sections of $\cF$ and $\Sc^*(X,\cF)$ denote its dual space.
\item Note that $\Sc(X_1,\cF_1) \otimes \Sc(X_2,\cF_2) \cong \Sc(X_1 \times X_2,\cF_1 \boxtimes \cF_2)$  for any $l$-spaces $X_i$ and sheaves $\cF_i$ on them.

\item An $l$-group is a topological group which has a basis of topology at $1$ consisting of open compact
subgroups. In fact, any topological group which is an $l$-space is an $l$-group.

\item Let an l-group $G$ acts (continuously) on an l-space $X$. Let $a:G \times X \to X$ be the action map and $p:G \times X \to X$ be the projection. A $G$-equivariant sheaf on $X$ is a sheaf $\cF$ on $X$ together with an isomorphism $a^{*} \cF \to p^{*} \cF$ which satisfy the natural conditions.

\item For a representation $V$ of an $l$-group $H$ we denote by $V^{\infty}$ the space of smooth vectors, i.e. vectors whose stabilizers
are open.

\item We denote $\widetilde{V}:=(V^*)^{\infty}$.
\item For an $l$-group $H$ we denote by $\cH(H)$ the convolution algebra of smooth (i.e. locally constant w.r.t. the action of $H$)  compactly supported measures on $H$.

\item Similarly for a transitive $H$-space $X$ we denote by $\cH(X)$ the space of smooth   compactly supported measures on $X$.

\item For an $l$-group $H$ we denote by $\cM(H)$ the category of smooth  representations of $H$.

\item Recall that if an $l$-group $H$ acts (continuously) on an $l$-space $X$ and $\cF$ is an $H$-equivariant sheaf on $X$ then $\Sc(X,\cF)$
is a smooth representation of $H$.
\end{itemize}

\begin{defn}
A representation $V$ of an $l$-group $H$ is called {\bf admissible} if one of the following equivalent conditions
holds.

\begin{enumerate}
\item For any  open compact subgroup $K<H$ we have $\dim V^K < \infty$.
\item There exists an open compact subgroup $K<H$ such that $\dim V^K < \infty$.
\item For any  open compact subgroup $K<H$, $V|_K =  \bigoplus \limits _{\rho \in \Irr K} n_{\rho} \rho$, where $n_{\rho}$ are finite numbers and  $\Irr K$ denotes the collection of isomorphism classes of irreducible representations of $K$.
\item The natural morphism $V \to \widetilde{\widetilde{V}}$ is an isomorphism.
\end{enumerate}
\end{defn}

\begin{thm}[Harish-Chandra]
Let $H$ be a reductive (not necessarily connected) group defined over $F$.
Then every smooth irreducible representation of $H(F)$ is admissible.
\end{thm}

\begin{defn}
Let $H$ be an $l$-group.
An $\cH(H)$-module $M$ is called {\bf unital} if $\cH(H)M=M$.
\end{defn}

\begin{thm}[Bernstein-Zelevinsky]\label{thm:RepUMod}
Let $H$ be an $l$-group. Then\\
(i) the natural functor between $\cM(H)$ and the category of unital $\cH(H)$-modules is an equivalence of
categories.\\
(ii) The category $\cM(H)$  is abelian.
\end{thm}

\subsection{Multiplicity free functors} \label{subsec:MultFreeFun}

\begin{defn}
Let $H$ be an $l$-group. We call a representation $\pi  \in \cM(H)$ {\bf multiplicity free} if for any irreducible
admissible representation $\tau \in \cM(H)$ we have $\dim_{\C} \Hom(\pi, \tau) \leq 1$.

Let $H'$ be an $l$-group. We call a functor $\cF:\cM(H) \to \cM(H')$ a {\bf multiplicity free functor} if for any irreducible
admissible representation $\pi \in \cM(H)$, the representation $\cF(\pi)$ is multiplicity free.
\end{defn}

\begin{rem}
Note that if $H$ is not reductive then the
"multiplicity free" property might be rather weak since there might be too few admissible representations.
\end{rem}

\begin{thm}\label{thm:FunctorIsModule}
Let $H$ and $H'$ be $l$-groups.\\
Let $\cF:\cM(H) \to \cM(H')$ be a $\C$-linear functor that commutes with arbitrary direct limits (or, equivalently, is right exact and commutes  with arbitrary direct sums). Let $\Pi:=\cF( \cH(H))$.
Consider the action of $H$ on $\cH(H)$ given by
$g\mu:=\mu*\de_{g^{-1}}$. It defines an action of $H$ on $\Pi$ which commutes with the action of $H'$. In this way $\Pi$ becomes a representation of
$H \times H'$. 
Then\\
(i) $\Pi$ is a smooth representation. \\
(ii)  $\cF$ is canonically isomorphic to the functor given by $\pi \mapsto
(\Pi \otimes \pi)_H$.
\end{thm}

This theorem is known. 
For the sake of completeness we include its prove in Appendix \ref{subapp:PfFunMod}.

\begin{thm}\label{thm:Mult1FunctorIsMult1Module}
Let $H$ and $H'$ be $l$-groups.\\
Let $\cF:\cM(H) \to \cM(H')$ be a $\C$-linear functor that commutes with arbitrary direct limits. 
Then $\cF$ is a multiplicity free functor if and only
if $\cF(\cH(H))$ is a multiplicity free representation of $H \times H'$.  
\end{thm}
For proof see Appendix \ref{subapp:PfFreeFunMod}.

\subsection{Gelfand Kazhdan criterion for "multiplicity free" property of geometric representations}  \label{subsec:GK}
\begin{thm}\label{thm:GK}
Let $H$ be an $l$-group.
Let $X$ and $Y$ be $H$-spaces and $\cF$ and $\cG$ be $H$-equivariant sheaves on $X$ and $Y$ respectively.
Let $\tau:X \to Y$ be a homeomorphism (not necessarily $H$-invariant).
Suppose that we are given an isomorphism $\tau_*\cF \simeq \cG$.
Define $T:X \times Y \to X \times Y$ by $T(x,y):=(\tau^{-1}(y),\tau(x))$.
It gives an involution $T$ on the space $\Sc^*(X \times Y,\cF \boxtimes \cG)$.

Suppose that any $\xi \in \Sc^*(X \times Y,\cF \boxtimes \cG)$ which is invariant
with respect to the diagonal action of $H$ is invariant with respect to $T$.
Then for any irreducible admissible representation $\pi \in \cM(H)$ we have
$$ \dim \Hom (\Sc(X,\cF), \pi) \cdot \dim \Hom (\Sc(Y,\cG)), \widetilde{\pi}) \leq 1.$$ 
\end{thm}
In the case when $X$ and $Y$ are transitive and correspond to each other in a certain way, this theorem is a classical theorem by Gelfand and Kazhdan (see \cite{GK}).
For the general case the proof is the same and we repeat it in Appendix \ref{app:GK}.
In fact, in this paper we could use the classical formulation of this theorem, but we believe that this theorem is useful in the general formulation.
%
\begin{defn}
Let $H$ be an $l$-group.
Let $\theta:H \to H$ be an involution. 
Let $X$ be an $H$-space. \\
(i) Denote by $\theta(X)$ the $H$-space which coincides with $X$ as an
$l$-space but with the action of $H$ twisted by $\theta$. \\
(ii) Similarly, for a representation $\pi$ of $H$ we denote by
$\theta(\pi)$ the
representation $\pi \circ \theta$.\\
(iii) Let $\cF$ be an $H$-equivariant sheaf on $X$. Let us define an equivariant sheaf $\theta(\cF)$ on $\theta(X)$. As a sheaf, $\theta(\cF)$ coincides with $\cF$ and the equivariant structure is defined in the following way. Let $a:H \times X \to X$ denote the action map and $p_2:H \times X \to X$ denote the projection.
Let $\alpha: a^*(\cF) \to p_2^*(\cF)$ denote the equivariant structure of $\cF$.
We have to define an equivariant structure $\theta(\alpha): (\theta(a))^*(\theta(\cF)) \to p_2^*(\theta(\cF))$, where $\theta(a):H \times
\theta(X) \to \theta(X)$ is the action map.
Note that $(\theta(a))^*(\theta(\cF)) \cong (\theta \times Id)^*(a^*(\cF))$. Since $\theta \times Id$ is an  involution, 
it is enough to define a map between $a^*(\cF)$ and $(\theta \times Id)^*(p_2^*(\cF))$. Let $\beta$ denote the canonical isomorphism
between $(\theta \times Id)^*(p_2^*(\cF))$ and $(p_2 \circ
(\theta\times Id))^*(\cF) = p_2^*(\cF)$. Now, the desired map is given by $\beta^{-1} \circ \alpha$.
\end{defn}

\begin{remark}
Clearly, $\Sc(\theta(X),\theta(\cF)) \cong \theta(\Sc(X,\cF)).$
\end{remark}

\begin{notn}
Let $H:=\GL_{n_1} \times ... \times \GL_{n_k}$. We denote by $\kappa$ the Cartan involution $\kappa(g)
:=(g^t)^{-1}$.
\end{notn}

\begin{thm}[\cite{GK}]\label{thm:DualKappa}
Let $H:=\GL_{n_1} \times ... \times \GL_{n_k}$. Let $\pi$ be an irreducible smooth representation of $H(F)$. Then 
$\widetilde{\pi} \simeq \kappa(\pi)$.
\end{thm}

\begin{cor}\label{cor:MultFree}
Let $H:=\GL_{n_1} \times ... \times \GL_{n_k}$. Let $X$ be an $H(F)$-space. Let $\cF$ be an $H(F)$-equivariant sheaf on $X$.
Suppose that any $\xi \in \Sc(X \times \kappa(X), \cF \boxtimes \kappa(\cF))$ which is invariant with respect to the diagonal action of $H(F)$ is invariant with respect to swap of the coordinates. Then the representation $\Sc(X, \cF )$ is multiplicity  free.
\end{cor}

\subsection{Bernstein-Gelfand-Kazhdan-Zelevinski criterion for vanishing of invariant distributions} \label{subsec:BGKZ}

\begin{thm}[Bernstein-Gelfand-Kazhdan-Zelevinsky] \label{thm:BGKZ}
Let an algebraic group $H$ act on an algebraic variety $X$, both defined over $F$. Let $H'$ be an open subgroup of $H(F)$.
Let $\cF$ be a sheaf over $X(F)$. Suppose that for any $x\in X(F)$
we have $$(\cF_x \otimes \Delta_{H'}|_{H'_x} \otimes  \Delta_{H'_x} ^{-1})^{H'_x}=0.$$
Then $\Sc(X, \cF )^{H'}=0.$
\end{thm}
This theorem follows from \cite[\S 6]{BZ} and \cite[\S\S 1.5]{Ber}.

\begin{cor}\label{cor:BGKZ}

Let an algebraic group $H$ act on an algebraic variety $X$, both defined over $F$.
Let $\sigma:X \to X$ be  an involution defined over $F$. Suppose that $\sigma$ normalizes the action of $H$.
Then each $H(F)$-invariant distribution on $X$ is invariant under $\sigma$.
\end{cor}

\subsection{Deligne filtration} \label{subsec:DelFilt}


\begin{thm}[Deligne]
Let $A$ be a nilpotent operator on a vector space $V$. Then there exists and unique a finite decreasing filtration $V^{\geq i}$ s.t. 
\\
(i) $A$ is of degree 2 w.r.t. this filtration.\\
(ii) $A^l$ gives an isomorphism $V^{\geq l}/V^{\geq l+1} \simeq V^{\geq -l}/V^{\geq -l+1}$.
\end{thm}

For proof see \cite[Proposition I.6.I]{Del} 

\begin{defn}
We will denote this filtration by $\cD_A^{\geq i}(V)$ and call it the Deligne filtration.
\end{defn}

\begin{notn}
The filtration $\cD_A^{\geq i}(V)$ induces filtrations on $\Ker A$ and $\Coker A$ in the following way
$$\cD_{A,+}^{\geq i}(\Ker A):= \cD_A^{\geq i}(V) \cap \Ker A \quad \text{ and }
\quad \cD_{A,-}^{\leq i}(\Coker A):= \cD_A^{\leq -i}(V)/ (\Im A \cap \cD_A^{\leq -i}(V)).$$
Denote by $\mu_A: \Gr^i_{A,-}(\Coker A) \to \Gr^i_{A,+}(\Ker A)$ the isomorphism given by $A^i$.
\end{notn}

\section{Implications between the main results} \label{sec:ImpRes}
\setcounter{lemma}{0}
In this section we prove that Theorems \ref{thm:PiRho},\ref{thm:Sc} and \ref{thm:GP} are equivalent and imply Theorem \ref{thm:GL}.

\begin{proof}[Proof that Theorem \ref{thm:PiRho} $\Leftrightarrow$ Theorem \ref{thm:Sc}]
Note that $J_M^G(\cH(G)) \cong \cH(U\backslash G)$ where  the action of $M$ is from the left and the action of $G$ is from the right. Clearly this representation of $G \times M$ is isomorphic  to the representation $\cH(G/U)$  that was described in Theorem \ref{thm:Sc}.
The equivalence follows now from Theorem \ref{thm:Mult1FunctorIsMult1Module}.
\end{proof}

\begin{proof}[Proof that Theorem \ref{thm:Sc} $\Leftrightarrow$ Theorem \ref{thm:GP}]
Note that $(G \times M)/P = G/U$. Hence $\cH(G/U) = \cH((G \times M)/P)$.
Now 
\begin{multline*}
Hom_{G \times M}(\cH(G/U) ,\pi)=Hom_{G \times M}(\cH((G \times M)/P) ,\pi)=Hom_{G \times M}(\tilde \pi,C^{\infty}((G \times M)/P))=\\=Hom_{G \times M}(\tilde \pi,Ind_P^{G \times M}(\C))= Hom_{P}(\tilde \pi,\C).
\end{multline*}
\end{proof}

\begin{proof}[Proof that Theorem \ref{thm:PiRho} implies Theorem \ref{thm:GL}]
Note that the center $Z(G)$ of $G$ lies in $M$, and that
$M \cong Z(G) \times H$. Now, let $\pi$ be an irreducible representation of $G$. Then $Z(G)$ acts on it by a character $\chi$. Let $\rho$ be an irreducible representation of $H$. Extend it to a representation of $M$ by letting $Z(G)$ act by $\chi$. Then
$ \Hom_H(J(\pi),\rho) = \Hom_M(J(\pi),\rho),$ which is at most one dimensional by Theorem \ref{thm:PiRho}.
\end{proof}

\section{Reduction to the geometric statement}  \label{sec:RedGeo}
\setcounter{lemma}{0}

\begin{defn}
Let $X:=X_{n,k}:=\{A,B\in Mat_{n+k}| AB=BA=0,rank(A)=n,rank(B)=k\}$.  Let $G$ act on $X_{n,k}$ by conjugations.
We define the transposition map $\theta:=\theta_{n,k}:X_{n,k} \to X_{n,k}$ by $\theta(A,B):=(A^t,B^t)$.
\end{defn}

In this section we deduce Theorem \ref{thm:Sc} from the following geometric statement.

\begin{prop}[geometric statement]\label{thm:Geo}
Any $G$-orbit in  $X_{n,k}$ is $\theta$-invariant.
\end{prop}

\begin{defn}$ $\\
(i) We denote by $E_{n,k}$ the $l$-space of exact sequences of the form
$$ 0 \to F^n \overset{\phi}{\to} F^{n+k} \overset{\psi}{\to} F^k \to 0.$$
We consider the natural action of $G\times M$ on $E_{n,k}$ given by $$(g,(h_1,h_2))(\phi,\psi):= (g\phi h_1^{-1},h_2\psi g^{-1}).$$
(ii) We denote by $\tau:E_{n,k} \to E_{k,n}$ the map given by $\tau(\phi,\psi):=(\psi^t,\phi^t).$ \\
(iii) We denote by $T:E_{n,k} \times E_{k,n}  \to E_{n,k} \times E_{k,n}$ the map given by $T(e_1,e_2):=(\tau(e_2),\tau(e_1)).$
\end{defn}

The following lemma is straightforward.

\begin{lem} \label{lem:Geo}$ $\\
(i) $G/U \cong E_{n,k}$ as a $G\times M$ - space.\\
(ii) The transposition map $\tau$ defines an isomorphism of $G\times M$ - spaces $\tau:E_{n,k} \to \kappa(E_{k,n})$.
\end{lem}

\begin{notn}
Denote by $C_{n,k}:E_{n,k}\times E_{k,n} \to X_{n,k}$ the composition map given by
$$C_{n,k}((\phi_1,\psi_1),(\phi_2,\psi_2)):=(\psi_2 \circ \phi_1,\psi_1 \circ \phi_2).$$
\end{notn}

The following lemma is straightforward.

\begin{lem}
$ $\\
(i) $C_{n,k}$ defines a bijection between $G\times M$ -orbits on $E_{n,k}\times E_{k,n}$ and  $G$-orbits on $X_{n,k}$.\\
(ii) $C_{n,k} \circ T = \theta \circ C_{n,k}$ .
\end{lem}

\begin{cor}
The geometric statement implies that all $G\times M$ -orbits on $E_{n,k}\times E_{k,n}$ are $T$-invariant.
\end{cor}

\begin{cor}
The geometric statement implies Theorem \ref{thm:Sc}.
\end{cor}
This corollary follows from the previous corollary, Lemma \ref{lem:Geo}, Corollary \ref{cor:BGKZ} and Corollary \ref{cor:MultFree}.

\section{Proof of the geometric statement (Proposition \ref{thm:Geo})} \label{sec:PFGeoStat}
\setcounter{lemma}{0}

The proof is by induction on $n$. From now on we assume that the geometric statement holds for all dimensions
smaller than $n$.
\begin{rem}
The proof that will be given here is valid for any field $F$.
\end{rem}

We will use the following  lemma.
\begin{lemma}[Key Lemma]\label{lem:Key}
Let $G':=GL_k$. Let $P_+'$ be its parabolic subgroup. Let $P_-'$ be the opposite parabolic. Let $P''$ be the subgroup of $P_+'\times P_-'$ consisting of pairs with the same
Levi part. Consider the two sided action of $P_+'\times P_-'$  on $G$ 
(given by $(p_1,p_2)g:=p_1gp_2^{-1}$) and its restriction to $P''$.

Then any $P''$ orbit on $G'$ is transposition invariant.
\end{lemma}

Since $k \leq 2$, this lemma is a straightforward computation.
\begin{remark}
The analogous statement for $k \geq 3$ is not true. In fact, this lemma is the only place where we use the assumption $k \leq 2$.
\end{remark}

\begin{notn}
Denote $X':=X'_{n,k}:=\{(A,B) \in  X\, | \, A \text{ is nilpotent}\}$.
\end{notn}

We will show that the geometric statement follows from the following proposition.
\begin{prop}\label{prop:GeoNilp}
Any $G$-orbit in  $X'_{n,k}$ is $\theta$-invariant.
\end{prop}
\begin{proof}[Proof that Proposition \ref{prop:GeoNilp} implies Theorem \ref{thm:Geo}]
Let $(A,B) \in X -X'.$ We have to show that there exists $g \in G$ such that $gAg^{-1}=A^t$ and $gBg^{-1}=B^t$.

Decompose $F^{n+k}:=V \oplus W$ such that $A=A' \oplus A''$ where $A'$ is a nilpotent operator on $V$ and $A''$
is an invertible operator on $W$. Since $AB=BA=0$, we have $B = B' \oplus 0$, where $B'$  is an operator on $V$  and $0$ denotes the
zero operator on $W$. Without loss of generality we may assume that $V$ and $W$ are coordinate spaces.

By the induction assumption, there exists $g_1 \in \GL(V)$ such that $g_1A'g_1^{-1}=A'^t$ and $g_1B'g_1^{-1}=B'^t$.
It is well known that  there exists $g_2 \in \GL(W)$ such that $g_2A''g_2^{-1}=A''^t$. Take $g:=g_1\oplus g_2$.
\end{proof}

\begin{notn}
Let $A$ be a nilpotent operator on a vector space $V$. Let $\nu_A:\GL(V)_A \to \GL(Ker A)\times \GL(Coker A)$ denote the map defined by
$\nu_A(g):=(g|_{Ker A}, g|_{Coker A})$. Denote also
\begin{multline*}
\cP_A:= \{g,h \in \GL(Ker A)\times \GL(Coker A) \, | \, g \text{ preserves } \mathcal{D}_{A,+}, h \text{ preserves } \mathcal{D}_{A,-}
\text{ and }\\ Gr_{\mathcal{D}_{A,+}}(g) \text{ corresponds to } Gr_{\mathcal{D}_{A,-}}(h) \text{under the identification } \mu_A \}.
\end{multline*}
\end{notn}

\begin{lem} \label{lem:nuAPA}
Let $A$ be a nilpotent operator on a vector space $V$. Then $\Im(\nu_A) = \cP_A$.
\end{lem}

\begin{proof}
Clearly  $\Im(\nu_A) \subset \cP_A$.
Let $\mathfrak{p}$ denote the Lie algebra of $\cP_A$. It is enough to show that the map $d\nu_A:\gl(V)_A \to \mathfrak{p}$ is onto.
Let $V=\bigoplus V_i$ be the decomposition of $V$ to Jordan blocks w.r.t. the action of $A.$
We have
\begin{align}
& \gl(V)_A = (V^*\otimes V)^A = \bigoplus_{i,j} (V_i^* \otimes V_j)^A\\
& \gl(Ker A)= (V^A)^* \otimes V^A  = \bigoplus_{i,j} (V_i^A)^* \otimes V_j^A\\
& \gl(Coker A)= (V/AV)^* \otimes (V/AV)  = \bigoplus_{i,j} (V_i/AV_i)^* \otimes (V_j/AV_j)
\end{align}

The filtration $\mathcal{D}_{A,+}$ on $Ker A$ gives a natural filtration on $\gl(Ker A)$. It is easy to see that the 1-dimensional space $(V_i^A)^* \otimes V_j^A$ is of degree $\dim V_j - \dim V_i$ w.r.t. this filtration. Similarly $(V_i/AV_i)^* \otimes (V_j/AV_j)$ is of degree  $\dim V_i - \dim V_j$.

Hence $\mathfrak{p} = \bigoplus \mathfrak{p}_{ij},$ where $$\mathfrak{p}_{ij}=
\left \{ \begin{array}{llll} (V_i^A)^* \otimes V_j^A & \quad \dim V_j > \dim V_i\\
(V_i/AV_i)^* \otimes (V_j/AV_j) & \quad \dim V_j < \dim V_i\\
\{(X,Y) \in (V_i^A)^* \otimes V_j^A \oplus (V_i/AV_i)^* \otimes (V_j/AV_j) \, | \, X \text{ corresponds to }Y & \\
\text{ under the
identification given by } A^{\dim V_i -1}  \} & \quad \dim V_j = \dim V_i \\
\end{array} \right .$$
This 
decomposition
 gives a decomposition $d\nu_A = \bigoplus \nu_{ij}$, where $\nu_{ij}:(V_i^* \otimes V_j)^A \to \mathfrak{p}_{ij}$.
It is enough to show that $\nu_{ij}$ is surjective for any $i$ and $j$.
Choose a gradation on $V_i$ which is compatible with the Deligne filtration.  Let $L_{ij} \subset (V_i^* \otimes V_j)^A$ be the 1-dimensional subspace of vectors of weight  $\dim V_j - \dim V_i$ w.r.t. this gradation.
It is easy to see that $\nu_{ij}|_{L_{ij}}$ is surjective.
\end{proof}

The following lemma is a reformulation of the Key Lemma.
\begin{lem} \label{lem:DualKey}
Let $V$ and $W$ be linear spaces of dimension $k$. Suppose that we are given a non-degenerate pairing between $V$ and $W$. Let
$\mathcal{F}$ be a descending filtration on $V$ and $\mathcal{G}$ be the dual, ascending,  filtration on $W$.
Suppose that we are given an isomorphism of graded linear spaces $\mu : Gr_{\mathcal{F}}(V) \to  Gr_{\mathcal{G}}(W)$.
Let
\begin{multline*}
\cP:= \{g,h \in \GL(V)\times \GL(W) \, | \, g \text{ preserves } \mathcal{F}, h \text{ preserves } \mathcal{G}
\text{ and }\\ Gr_{\mathcal{F}}(g) \text{ corresponds to } Gr_{\mathcal{G}}(h) \text{under the identification } \mu \}.
\end{multline*}
Let $\cP$ act on $\Hom(V,W)$ by $(g,h)(\phi):= h \circ \phi \circ g^{-1}$. Note that the pairing between $V$ and $W$ defines a notion of
transposition on $\Hom(V,W)$.

Then any $\cP$-orbit on $\Hom(V,W)$ is invariant under transposition.
\end{lem}

\begin{proof}
[Proof of Proposition \ref{prop:GeoNilp}]
Let $(A,B) \in X'$. We have to show that there exists $g \in G$ such that $gAg^{-1}=A^t$ and $gBg^{-1}=B^t$.
Fix a bilinear form $Q$ on $F^{n+k}$ such that $A^t_Q=A$, where $A^t_Q$ denotes transpose with respect to the form $Q$. It is enough
to show that there exists $g \in G_A$ such that  $gBg^{-1} = B_Q^t$. Note that $\Ker A =\Im B$ and $\Ker B =\Im A$. Denote by
$B':\Coker A \to \Ker A$ the map induced by $B$. Consider the natural action of $\GL(\Coker A) \times \GL(\Ker A)$ on
$\Hom(\Coker A, \Ker A)$. 
Note that $\Ker B_Q^t = \Im A$ and $\Ker A = \Im B_Q^t$ and hence $B_Q^t$ also induces a map $\Coker A \to \Ker A$. Denote this map
by $B''$. Note that $B''$ is the transposition of the map $B'$ with respect to the non-degenerate pairing between $\Coker A$ and $\Ker A$ given by $Q$.
The assertion follows now from Lemma \ref{lem:nuAPA} and Lemma \ref{lem:DualKey}.

\end{proof}

\section{Discussion of the higher rank cases} \label{sec:HiRank}
\setcounter{lemma}{0}

In this section we discuss whether an analog of Theorem \ref{thm:PiRho} holds when $M$ is an arbitrary Levi subgroup.
If $F$ is a finite field, a negative answer to this question can be obtained from a  negative answer to an analogous question for permutation groups. We discuss permutation groups in \S\S \ref{subsec:Perm} and the connection between the two questions in \S\S \ref{subsec:Con}.
The answer we obtain is that
such  analog of Theorem \ref{thm:PiRho}  holds only in the cases at hand.

We discuss the case when $F$ is a local field in \S\S \ref{subsec:LocHighRank}, but we do not reach a conclusion.

Since the results here are negative and mostly known, the
discussion is rather informal and some details are omitted.

\subsection{The analogous problems for the permutation groups} \label{subsec:Perm}$ $

Let $M'=S_{n_1} \times ... \times S_{n_l}$ and $G':=S_{n_1+...+n_l} $. One can ask when $(G',M')$ is a strong Gelfand pair, i.e. when the restriction functor from $G'$ to $M'$ is multiplicity free. The answer is: $(G',M')$ is a strong Gelfand pair if and only if $l \leq 2$ and $\min(n_1,n_2) \leq 2$.
This is well known,
but
let us indicate the proof.

The fact that the pairs $(S_{n+1},S_{n})$ and $(S_{n+2},S_{n}\times S_2)$ are strong Gelfand pairs follows by Theorems \ref{thm:Mult1FunctorIsMult1Module} and \ref{thm:GK} from the fact that every permutation from $G'$ is conjugate by $M'$ to its inverse.

In order to show that other pairs mentioned above are not strong Gelfand pairs, 
we have to show that the algebra of $Ad(M')$-invariant functions on $G'$ with respect to convolution is not commutative unless $l \leq 2$ and $\min(n_1,n_2) \leq 2$.

If $l \geq 3$ then consider the transpositions $\sigma_1=(1,n_1+1)$ and $\sigma_2=(n_1+1,n_2+1)$. It is easy to see that the characteristic functions of their $M'$-conjugacy classes do not commute.
If $l = 2$ and $n_1,n_2 \geq 3$ then consider the cyclic permutations $\sigma_1=(1,2,3,n_1+1,n_1+2,n_1+3)$ and $\sigma_2=(1,n_1+1,n_1+2)$. It is easy to see that the characteristic functions of their $M'$-conjugacy classes do not commute.

\subsection{Connection with our problem for the finite fields} \label{subsec:Con}$ $
Suppose that $F$ is a finite field.
Let $M=GL_{n_1}(F) \times ... \times GL_{n_l}(F)$ and $G:=GL_{n_1+...+n_l}(F)$. Then the multiplicities problem of Jacquet functor between $\cM(G)$ and $\cM(M)$ can be considered as a generalization of a deformation of the multiplicities problem of the restriction functor from $\cM(G')$ to $\cM(M')$.

Indeed, the multiplicities problem of Jacquet functor is equivalent to multiplicities problem of the parabolic induction from $\cM(M)$ to $\cM(G)$. Let $\Sigma:=i_{T_M}^M(\C)$, where $i_{T_M}^M$ denotes the parabolic induction from the torus of $M$ to $M$.
Let $\Pi:=i_{T_G}^G(\C)$. Let $\cA$ be the subcategory of
$\cM(M)$
generated by $\Sigma$ and $\cB$ be the subcategory of
$\cM(G)$
generated by $\Pi$. Then the multiplicities problem of the parabolic induction from $\cA$ to $\cB$ is a special case of the multiplicities problem of the parabolic induction from $\cM(M)$ to $\cM(G)$.
Let $A:=End_{M}(\Sigma)$ and $B:=End_{G}(\Pi)$. Clearly, $\cA$ is equivalent to the category of $A$-modules and $\cB$ is equivalent to the category of $B$-modules. It is well known that $A$ and $B$ are deformations of the group algebras of
$M'$ and $G'$
respectively.
Therefore the multiplicities problem of the parabolic induction from $\cA$ to $\cB$  is a deformation of  the multiplicities problem of the induction from $M'$ to $G'$, which in turn is equivalent to he multiplicities problem of the restriction from $G'$ to $M'$.
In fact, one can show that those deformations are trivializable since those algebras are semisimple.

One can use this argumention in order to show that $(G',M')$ is a strong Gelfand pair only if $l \leq 2$ and $\min(n_1,n_2) \leq 2$.

\subsection{Higher rank cases over local fields} \label{subsec:LocHighRank}$ $

First note that the reduction of Theorem \ref{thm:Sc} to the Key Lemma works without change for arbitrary $k$.
This reduction connects between the Gelfand-Kazhdan criterion for the "multiplicity free" property of the  Jacquet functor from $\GL_{n+k}(F)$ to $\GL_{n}(F) \times \GL_{k}(F)$ and  the Gelfand-Kazhdan criterion for the "multiplicity free" property of the  Jacquet functor from $\GL_{k}(F)$ to an arbitrary Levi subgroup. Therefore we believe that the  "multiplicity free" properties themselves are connected and if one wants to consider the case of arbitrary $k$, he will also have to consider arbitrary Levi subgroups. At the moment we do not have an opinion when the Jacquet functor from $\GL_n(F)$ to an arbitrary Levi subgroup is multiplicity free.
\appendix

\section{Multiplicity free functors } \label{app:MultFree}

\subsection{Proof of Theorem \ref{thm:FunctorIsModule}} \label{subapp:PfFunMod}

\begin{proof}[Proof of Theorem \ref{thm:FunctorIsModule}, (i)]

 For any open compact subgroup $K<H$ denote by $\Sigma_K \subset \cH(H)$ the subspace of right $K$-invariant measures.
Denote $\Pi_K:=\cF(\Sigma_K)$. Since $\cF$ commutes with direct limits, $\lim_K \Pi_K \cong \Pi$. It is easy to see that $K$ acts trivially on the image of $\Pi_K$ in $\Pi$. Hence $\Pi$ is a smooth representation of
$H$ and hence it is a smooth representation of
$H \times H'$.
%
%
%
\end{proof}

For the proof of (ii) we will need several lemmas.

\begin{notn}
Denote by $\cH(H)_0$ the subalgebra of $\cH(H)$ consisting of functions with zero integral.
\end{notn}

\begin{lem} \label{lem:Coinv}
Let $\pi$ be a smooth representation of
$H$. Then $\pi_H = \Coker(\cH(H)_0\otimes \pi \to \pi)$, where by equality we mean equality of quotients of $\pi$.
\end{lem}
\begin{proof}
Let $V$ be any vector space. We can consider it as a representation of $H$ with trivial action or as a $\cH(H)$-module on which every measure acts by multiplication by its integral. Then
$$\Hom_{\C}(\pi_H,V) = \Hom_{H}(\pi,V) \text{ and } \Hom_{\C}(\Coker(\cH(H)_0\otimes \pi \to \pi),V) = \Hom_{\cH(H)}(\pi,V).$$
By Theorem \ref{thm:RepUMod}, $\Hom_{H}(\pi,V) = \Hom_{\cH(H)}(\pi,V)$ and therefore $$\Hom_{\C}(\pi_H,V) = \Hom_{\C}(\Coker(\cH(H)_0\otimes \pi \to \pi),V)$$ for any vector space $V$. The lemma follows now from the Yoneda lemma.
\end{proof}

\begin{lem}
Let $\pi$ be a smooth representation of
$H$. Let $H$ act on 
$\cH(H)\otimes \pi$ by $g(\mu \otimes v):=(\mu * \de_{g^{-1}})\otimes gv$.

Then $(\cH(H)\otimes \pi)_H = \pi$, where by equality we mean equality of quotients of $\cH(H)\otimes \pi$.
\end{lem}
\begin{proof}
Let us deduce the statement from the Yoneda lemma. Let $\tau$ be a smooth representation of $H$. Then
\begin{multline*}
\Hom_H((\cH(H) \otimes \pi)_H,\tau) = \Hom_{H\times H}(\cH(H) \otimes \pi,\tau)= (\Hom_\C(\cH(H) \otimes \pi,\tau))^{H\times H}=\\
=(\Hom_\C(\cH(H), \Hom_{\C}(\pi,\tau)))^{H\times H}=(\Hom_{\cH(H)}(\cH(H), \Hom_{\C}(\pi,\tau)))^{H}= ( \Hom_{\C}(\pi,\tau))^{H}=\Hom_H(\pi,\tau)
\end{multline*}

\end{proof}

\begin{cor}
The following sequence is exact
$$ \cH(H)_0\otimes \cH(H) \otimes \pi \to \cH(H)\otimes \pi \to \pi \to 0.$$
\end{cor}

\begin{proof}[Proof of Theorem \ref{thm:FunctorIsModule}, (ii)]
Let $H$ act on $\cH(H)_0\otimes \cH(H) \otimes \pi$ and $\cH(H) \otimes \pi$ by acting on the $\cH(H)$ component.
Consider the exact sequence of $H$-representations
$$ \cH(H)_0\otimes \cH(H) \otimes \pi \to \cH(H)\otimes \pi \to \pi \to 0.$$
Since $\cF$ is right exact, the sequence
$$ \cF(\cH(H)_0\otimes \cH(H) \otimes \pi) \to \cF(\cH(H)\otimes \pi) \to \cF(\pi) \to 0$$
is exact.

Since $\cF$ commutes with direct sums, the later sequence is isomorphic to
$$ \cH(H)_0\otimes \Pi \otimes \pi \to \Pi \otimes \pi \to \cF(\pi) \to 0.$$
The theorem follows now from Lemma \ref{lem:Coinv}.
\end{proof}

\subsection{Proof of Theorem \ref{thm:Mult1FunctorIsMult1Module}}\label{subapp:PfFreeFunMod}$ $

The following lemma is standard.
\begin{lemma}
Let $K$ be a compact $l$-group and $L<K$ be an open subgroup. Then there is a finite number of isomorphism classes of irreducible representations of $K$ which have an $L$-invariant vector.
\end{lemma}

\begin{cor}
Let $K$ be a compact $l$-group. Let $\pi = \prod \pi_{\sigma}$ be a product of smooth isotypic components of $K$.
Then $\pi^{\infty}=\bigoplus \pi_{\sigma}.$
\end{cor}

\begin{cor}
Let $H$ be an $l$-group. Let $\pi$ and $\rho$ be smooth admissible representations of $H$. Then $$\Hom_{\C}(\pi,\rho)^{\infty} = \widetilde{\pi} \otimes \rho.$$
\end{cor}

\begin{cor}
Let $H$ and $H'$ be $l$-groups.
Let $\cF:\cM(H) \to \cM(H')$ be a $\C$-linear right exact functor. Let $\Pi:=\cF(\cH(H))$.
Let $\pi$ and $\rho$ be smooth admissible representations of $H$ and $H'$ respectively.
Then $$\Hom_{H\times H'}(\Pi,\widetilde{\pi} \otimes \rho) = \Hom_{H'}(\cF(\pi),\rho).$$
\end{cor}
\begin{proof}
\begin{multline*}
\Hom_{H'}(\cF(\pi),\rho) = \Hom_{H'}((\Pi \otimes \pi)_H, \rho) =  \Hom_{\C}((\Pi \otimes \pi)_H, \rho)^{H'}=(\Hom_{\C}(\Pi \otimes \pi, \rho)^H)^{H'}=\\=  \Hom_{\C}(\Pi \otimes \pi, \rho)^{H \times {H'}}
= \Hom_{\C}(\Pi, \Hom_{\C}(\pi, \rho))^{H \times {H'}}= \Hom_{H \times {H'}}(\Pi, \Hom_{\C}(\pi, \rho))=\\
=\Hom_{H \times {H'}}(\Pi, \Hom_{\C}(\pi, \rho)^{\infty}) = \Hom_{H\times {H'}}(\Pi,\widetilde{\pi} \otimes \rho)
\end{multline*}
\end{proof}
\begin{cor}
Theorem \ref{thm:Mult1FunctorIsMult1Module} holds.
\end{cor}

\section{Proof of Theorem \ref{thm:GK}} \label{app:GK}
\setcounter{lemma}{0}
We will use the following classical well-known lemma.
\begin{lem}
Let $H$ be an $l$-group and $\pi$ be an irreducible admissible representation of $H$. \\
(i) Let $\rho \in \cM(H)$ and $\phi: \rho \to \pi$. Then $\phi$ is an epimorphism.\\
(ii) Let $v \in \pi$. If $\psi(v)=0$ for any $\psi \in \widetilde{\pi}$ then $v = 0$.\\
(iii) $\dim \Hom(\pi,\pi)=1$.
\end{lem}

\begin{proof}[Proof of Theorem \ref{thm:GK}]
If $\Hom (\Sc(X,\cF), \pi)=0$ we are done. Otherwise let $\phi \in \Hom (\Sc(X,\cF), \pi) - 0$.
Let $\psi_1, \psi_2 \in \Hom (\Sc(Y,\cG), \widetilde{\pi})$. Let us show that they are dependent.
If one of them is zero we are done, so we assume the contrary.

Define bilinear forms $\xi_i:\Sc(X,\cF)\otimes \Sc(Y,\cG) \to \C$ by $$\xi_i(f\otimes h):=\langle \psi_i(h),\phi(f) \rangle .$$
Let $V_{i}$ be left kernels of $\xi_i$, i.e.
$$V_i=\{f \in \Sc(X,\cF) \, | \, \forall h\in \Sc(Y,\cG). \, \xi_i(f\otimes h)=0\}.$$
By the previous lemma, $V_i = \Ker \phi$ and hence $V_1 = V_2$. Let  $W_i$ be the right kernels of $\xi_i$. Again, the previous lemma implies that $W_i = \Ker \psi_i$. Now, consider $\xi_i$ as elements of $\Sc^*(X \times Y,\cF \boxtimes \cG)$. Clearly they are $H$-invariant. Hence, by the assumption of the theorem, $\xi_i$ are invariant with respect to $T$. Hence $W_i = \tau_*V_i$. Hence $W_1=W_2$ and by the previous lemma $\psi_1$ is proportional to $\psi_2$. This implies that $\dim \Hom (\Sc(Y,\cG)), \widetilde{\pi}) \leq 1$. Similarly  $\dim \Hom (\Sc(X,\cF)), \widetilde{\pi}) \leq 1$.
\end{proof}


\begin{thebibliography}{99}
\bibitem[\href{http://arxiv.org/abs/0910.3199}{AAG}]{AAG} Aizenbud, A.; Avni, N.; Gourevitch, D.:{\it Spherical pairs over close local fields}, arXiv:0910.3199[math.RT].

\bibitem[\href{http://arxiv.org/abs/0812.5063}{AG09a}]{AG_HC} Aizenbud, A.; Gourevitch, D.:{\it Generalized Harish-Chandra descent, Gelfand pairs and an
Archimedean analog of Jacquet-Rallis' Theorem.} Duke Mathematical Journal,  Volume 149, Number 3 (2009). See also arxiv:0812.5063v3[math.RT].
\bibitem[\href{http://arxiv.org/abs/0808.2729v1}{AG09b}]{AG_AMOT} A. Aizenbud,  D. Gourevitch, {\it Multiplicity one theorem for $(GL_{n+1}(\rr),GL_{n}(\rr))$}, Selecta Mathematica, Vol. 15, No. 2., pp. 271-294 (2009).
See also arXiv:0808.2729v1 [math.RT].
\bibitem[\href{http://arxiv.org/abs/0709.4215}{AGRS}]{AGRS} A. Aizenbud,  D. Gourevitch, S. Rallis, G. Schiffmann, {\it
Multiplicity One Theorems}, arXiv:0709.4215[math.RT], To appear
in the Annals of Mathematics.

\bibitem[\href{http://www.math.tau.ac.il/~bernstei/Publication_list/publication_texts/Bernstein-P-invar-SLN.pdf}{Ber83}]{Ber} J.
Bernstein, {\it $P$-invariant Distributions on $\mathrm{GL}(N)$
and the classification of unitary representations of
$\mathrm{GL}(N)$ (non-archimedean case),} Lie group
representations, II (College Park, Md., 1982/1983), 50--102,
Lecture Notes in Math., \textbf{1041}, Springer, Berlin (1984).


\bibitem[\href{http://www.math.tau.ac.il/~bernstei/Publication_list/publication_texts/Bernst_Lecture_p-adic_repr.pdf}{Ber}]{Ber_Lec} J.
Bernstein, {\it Representations of p-adic groups,} lecture notes written by Karl E. Rumelhart at Harvard University, Fall 1992. Available at\\ \url{http://www.math.tau.ac.il/~bernstei/Publication_list/publication_texts/Bernst_Lecture_p-adic_repr.pdf}.



\bibitem[\href{http://www.math.tau.ac.il/~bernstei/Publication_list/publication_texts/B-Zel-RepsGL-Usp.pdf}{BZ76}]{BZ} J. Bernstein, A.V.
Zelevinsky, {\it Representations of the group $\mathrm{GL}(n, F)$,
where F is a local non-Archimedean field,} Uspekhi Mat. Nauk
\textbf{10}, No.3, 5-70 (1976).

\bibitem[Del80]{Del} P.Deligne {\it La conjecture de Weil. II.} Inst. Hautes Études Sci. Publ. Math. {\bf 52} (1980), pp 137-252.

\bibitem[Fad78]{Fad} D. K. Faddeev:  {\it Complex representations of the general linear group over a finite field. Modules and representations,} Zap. Nauchn. Sere. LOMI, {\bf 46}, pp 64-88 (1974); English transl, in J. Soviet Math., {\bf 9}, No. 3, 64-88 (1978).
\bibitem[Gre]{Gre} J. A. Green,  {\it The characters of the finite general linear groups}, Trans. Amer. Math. Soc., {\bf 80}, pp 402-447 (1955).

\bibitem[GK71]{GK} I.M. Gelfand, D. Kazhdan, {\it Representations of the group ${\rm GL}(n,K)$ where $K$ is a local
field}, Lie groups and their representations (Proc. Summer School,
Bolyai Janos Math. Soc., Budapest, 1971), pp. 95--118. Halsted,
New York (1975).

\bibitem[\href{http://archive.numdam.org/ARCHIVE/CM/CM_1996__102_1/CM_1996__102_1_65_0/CM_1996__102_1_65_0.pdf}{JR96}]{JR}
H. Jacquet, S. Rallis, {\it Uniqueness of linear periods.},
Compositio Mathematica , tome \textbf{102}, n.o. 1 , p. 65-123
(1996).

\bibitem[\href{http://arxiv.org/abs/0903.1413}{SZ}]{SZ} B. Sun and C.-B. Zhu, {\it
Multiplicity one theorems: the archimedean case},
arXiv:0903.1413[math.RT].



%





\bibitem[Zel81]{Zel}
A. V. Zelevinsky: {\it Representations of finite classical groups}, Lect. Notes in Math., {\bf 869}, Springer-
Verlag, 1981.



\end{thebibliography}
\end{document}